\documentclass[12pt]{article}

\usepackage{amsmath,amssymb,amsfonts,latexsym}
\usepackage{graphicx}
\usepackage{caption}
\usepackage{booktabs}
\usepackage{cite}
\usepackage[colorlinks=true,linkcolor=blue,citecolor=blue,urlcolor=black]{hyperref}

\newtheorem{thm}{Theorem}
\newtheorem{prop}{Proposition}[section]

\newtheorem{lem}[prop]{Lemma}

\newcommand{\pf}{\noindent{\it Proof.} }
\newcommand{\qed}{{\hfill$\square$}\medskip}

\def\Z{{\mathbb Z}}
\def\Q{{\mathbb Q}}

\numberwithin{equation}{section}
\allowdisplaybreaks

\begin{document}

\begin{center}
{\Large\bf Proofs of two $q$-congruence conjectures of Guo}
\end{center}

\vskip 2mm \centerline{Ji-Cai Liu and Qing-Yuan Tao}
\begin{center}
{\footnotesize Department of Mathematics, Wenzhou University, Wenzhou 325035, PR China\\[5pt]
{\tt jcliu2016@gmail.com, 15514228254@163.com} \\[10pt]
}
\end{center}

\vskip 0.7cm \noindent{\bf Abstract.}
We prove two conjectural $q$-congruences proposed by Guo. The first
is Conjecture 7.2 in Guo's work on $q$-analogues of two ``divergent''
Ramanujan-type supercongruences; it asserts a square-cyclotomic congruence for a
truncated $q$-analogue of a Ramanujan-type sum when $n\equiv1\pmod4$.  The
second is Conjecture 4.1 in Guo's extension of Van Hamme's $(A.2)$
supercongruence; it gives divisibility modulo $[n]$ for a family of truncated
basic hypergeometric sums with a parameter $s$.  The proof of the first result
relies on a known Watson-transformation congruence obtained by Guo.  The proof of the second result is based on
period decomposition at primitive roots of unity and a reflection cancellation
inside residue blocks.


\vskip 2mm
\noindent{\it MR Subject Classifications}: Primary 33D15; Secondary 11A07,
11B65.

\section{Introduction}

Ramanujan-type supercongruences and their $q$-analogues form an active part of
the study of truncated hypergeometric and basic hypergeometric series.  Van
Hamme~\cite{VanHamme1997} proposed a list of supercongruences for truncated
Ramanujan-type series, and subsequent work of Guillera and
Zudilin~\cite{GuilleraZudilin2012}, Sun~\cite{Sun2011}, Guo and
Zudilin~\cite{GuoZudilin2019}, and others showed that many of these congruences
admit refined cyclotomic $q$-analogues.  Such $q$-congruences replace ordinary
prime-power divisibility by divisibility by powers of cyclotomic polynomials;
when $q\to1$ and the modulus is specialized at prime values, they recover
classical supercongruences.

We shall use the following standard notation.  For $m\ge1$,
\begin{equation*}
  [m]=[m]_q=\frac{1-q^m}{1-q},
\end{equation*}
and, for $N\ge0$,
\begin{equation*}
  (a;q)_0=1,\qquad (a;q)_N=\prod_{j=0}^{N-1}(1-aq^j).
\end{equation*}
We also write
\begin{equation*}
  (a_1,a_2,\ldots,a_r;q)_N=(a_1;q)_N(a_2;q)_N\cdots(a_r;q)_N.
\end{equation*}
The $n$-th cyclotomic polynomial is defined by
\[
  \Phi_n(q)=\prod_{\substack{1\le j\le n\\ \gcd(j,n)=1}}(q-\zeta^j),
\]
where $\zeta$ is a primitive $n$-th root of unity.  Congruences of
rational functions are interpreted in the usual reduced-denominator sense: if
$F(q)=P(q)/Q(q)$ with $P,Q\in\Q[q]$, $\gcd(P,Q)=1$, and $Q$ is relatively prime
to the stated modulus, then
\begin{equation*}
  F(q)\equiv0\pmod{M(q)}
\end{equation*}
means that $M(q)$ divides $P(q)$ in $\Q[q]$.

The first conjecture treated here arose in Guo's study of $q$-analogues of two
``divergent'' Ramanujan-type supercongruences.  Guo~\cite[Conjecture 7.2]{GuoDivergent}
proposed the following square-cyclotomic congruence.  We prove
it as the first main theorem of the present paper.

\begin{thm}\label{thm:guo72}
Let $n>1$ be an integer with $n\equiv1\pmod4$.  Then
\begin{equation}\label{eq:thm1-main}
\sum_{k=0}^{(n-1)/2} [4k+1]
   \frac{(q;q^2)_k^3}{(q^2;q^2)_k^3}
   q^{k(n^2-2nk-n-2)/4}
   \equiv 0 \pmod{\Phi_n(q)^2}.
\end{equation}
\end{thm}

The congruence in \eqref{eq:thm1-main} is a $q$-analogue of a
truncated Ramanujan-type supercongruence.  Indeed, as $q\to1$,
\begin{equation*}
  [4k+1]\to4k+1,
  \qquad
  \frac{(q;q^2)_k}{(q^2;q^2)_k}
  =\prod_{j=0}^{k-1}\frac{1-q^{2j+1}}{1-q^{2j+2}}
  \to\frac{(1/2)_k}{k!},
\end{equation*}
and, for prime $n=p$, $\Phi_p(1)=p$.  Thus the prime case of
\eqref{eq:thm1-main} specializes to the corresponding congruence modulo $p^2$, namely
\[
  \sum_{k=0}^{(p-1)/2}(4k+1)\frac{(1/2)_k^3}{k!^3}
  \equiv 0 \pmod{p^2}
  \qquad(p\ \text{prime},\ p\equiv1\pmod4).
\]
The proof in Section~\ref{sec:proof-thm1} uses a square-modulus auxiliary
congruence due to Guo~\cite{Guo2025} together with a root-of-unity pairing
argument.

The second conjecture appears in Guo's work on a new extension of Van Hamme's
$(A.2)$ supercongruence~\cite{GuoVanHamme}.  For $0\le s\le(n-1)/2$ and
$k\ge s$, define
\begin{equation*}
\mathcal T_{s,k}(q)=
(-1)^k[4k+1]q^k
\frac{(q;q^2)_{k-s}(q;q^2)_{k+s}(q;q^2)_k^2(q^2;q^4)_k}
{(q^2;q^2)_{k-s}(q^2;q^2)_{k+s}(q^2;q^2)_k^2(q^4;q^4)_k}.
\end{equation*}
The second main theorem proves a conjecture due to Guo~\cite[Conjecture 4.1]{GuoVanHamme}.

\begin{thm}\label{thm:guo41}
Let $n>1$ be odd and let $0\le s\le(n-1)/2$.  If
\begin{equation}\label{eq:M-values}
  M\in\left\{\frac{n-1}{2},\, n-s-1,\, n-1\right\},
\end{equation}
then
\begin{equation}\label{eq:thm2-main}
  \sum_{k=s}^{M}\mathcal T_{s,k}(q)
  \equiv 0 \pmod{[n]}.
\end{equation}
\end{thm}

The rest of the paper is organized as follows.  Section~\ref{sec:proof-thm1}
proves Theorem~\ref{thm:guo72}.  Section~\ref{sec:proof-thm2} proves
Theorem~\ref{thm:guo41} by reducing the congruence modulo $[n]$ to cancellation
at primitive roots of unity of every order $d\mid n$.

\section{Proof of Theorem~\ref{thm:guo72}}\label{sec:proof-thm1}

We start with an auxiliary congruence, which is the case $m=0$ of Guo~\cite[Theorem 2.1, (2.2)]{Guo2025}.
\begin{lem}\label{prop:guo-aux}
Let $n>1$ and $n\equiv1\pmod4$.  Then
\begin{equation*}
\sum_{k=0}^{(n-1)/2} [4k+1]_{x^2}
   \frac{(x^2;x^4)_k^3}{(x^4;x^4)_k^3}x^{-k}
   \equiv 0 \pmod{\Phi_n(x)^2}.
\end{equation*}
\end{lem}

We also record the elementary fact that squaring permutes primitive roots of odd
order and preserves multiplicities.

\begin{lem}\label{lem:squaring}
Let $n$ be an odd positive integer, and let $f(q)\in\Q[q]$.  If $f(q^2)$ is relatively prime to
$\Phi_n(q)$, then $f(q)$ is relatively prime to $\Phi_n(q)$.  If
\begin{equation}\label{eq:square-divisibility}
  \Phi_n(q)^2\mid f(q^2),
\end{equation}
then $\Phi_n(q)^2\mid f(q)$.
\end{lem}

\pf
Since $n$ is odd, the map $\zeta\mapsto\zeta^2$ is a bijection on the set of
primitive $n$-th roots of unity.  The coprimality assertion follows immediately
from this bijection.  For the divisibility assertion, let $\alpha$ be a primitive
$n$-th root of unity and choose a primitive $n$-th root $\zeta$ with
$\zeta^2=\alpha$.  If the multiplicity of $\alpha$ as a root of $f(q)$ is $m$,
then the multiplicity of $\zeta$ as a root of $f(q^2)$ is also $m$, because
$q^2-\alpha=(q-\zeta)(q+\zeta)$ and $q+\zeta$ does not vanish at $q=\zeta$.
Thus \eqref{eq:square-divisibility} forces $m\ge2$.  Since $\alpha$ was
arbitrary, $\Phi_n(q)^2\mid f(q)$.
\qed

The next lemma is the cancellation mechanism that supplies the second
cyclotomic factor in Theorem~\ref{thm:guo72}.

\begin{lem}\label{lem:thm1-pairing}
Let $n=4h+1$.  For $0\le k\le2h$, put
\begin{equation}\label{eq:Ak-def}
  A_k(x)= [4k+1]_{x^2}
   \frac{(x^2;x^4)_k^3}{(x^4;x^4)_k^3}x^{-k}.
\end{equation}
If $\alpha$ is a primitive $n$-th root of unity, then
\begin{equation}\label{eq:Ak-middle-zero}
  A_h(\alpha)=0,
\end{equation}
and
\begin{equation}\label{eq:Ak-pair}
  A_{2h-k}(\alpha)=-A_k(\alpha)\qquad (0\le k\le h-1).
\end{equation}
\end{lem}

\pf
All denominators in the values $A_r(\alpha)$ with $0\le r\le2h$ are nonzero:
they are products of factors $1-\alpha^{4s}$ with $1\le s\le r\le2h$, and
$n\mid4s$ is impossible because $n$ is odd and $0<s<n$.  Since
$[4h+1]_{\alpha^2}=[n]_{\alpha^2}=0$, \eqref{eq:Ak-middle-zero} follows.

Let $0\le k\le h-1$ and put $j=2h-k$.  Set $\beta=\alpha^2$, which is again a
primitive $n$-th root of unity.  Since $4j+1=2n-(4k+1)$, we have
\begin{equation}\label{eq:qint-ratio-thm1}
 \frac{[4j+1]_{\alpha^2}}{[4k+1]_{\alpha^2}}
 =\frac{1-\beta^{2n-(4k+1)}}{1-\beta^{4k+1}}
 =-\beta^{-(4k+1)}=-\alpha^{-8k-2}.
\end{equation}
It remains to compare the quotient of the shifted-factorial parts.  Define
\begin{equation*}
  P_k=\prod_{i=k}^{j-1}\frac{1-\alpha^{4i+2}}{1-\alpha^{4i+4}}
      =\prod_{i=k}^{j-1}\frac{1-\beta^{2i+1}}{1-\beta^{2i+2}}.
\end{equation*}
The denominator exponents are
\begin{equation*}
  2k+2,\,2k+4,\ldots,4h-2k,
\end{equation*}
and these are precisely the complements modulo $n=4h+1$ of the numerator
exponents $2i+1$, $k\le i\le j-1$, because
\begin{equation*}
  n-(2i+1)=4h-2i.
\end{equation*}
Therefore
\begin{equation*}
  P_k=\prod_{i=k}^{j-1}\frac{1-\beta^{2i+1}}{1-\beta^{-(2i+1)}}
      =\prod_{i=k}^{j-1}(-\beta^{2i+1}).
\end{equation*}
There are $j-k=2h-2k$ factors, an even number, and
\begin{equation*}
  \sum_{i=k}^{j-1}(2i+1)=4h(h-k).
\end{equation*}
Hence
\begin{equation}\label{eq:Pk-value}
  P_k=\beta^{4h(h-k)}=\alpha^{8h(h-k)}=\alpha^{2h+1+2k},
\end{equation}
where the last equality uses
\begin{equation*}
  8h(h-k)-(2h+1+2k)=(4h+1)(2h-1-2k).
\end{equation*}
Combining \eqref{eq:qint-ratio-thm1} and \eqref{eq:Pk-value}, and including the
factor $x^{-r}$ in \eqref{eq:Ak-def}, gives
\begin{equation}\label{eq:Ak-quotient}
\frac{A_j(\alpha)}{A_k(\alpha)}
 =\left(-\alpha^{-8k-2}\right)
   \left(\alpha^{2h+1+2k}\right)^3
   \alpha^{-(j-k)}.
\end{equation}
The exponent of $\alpha$ in \eqref{eq:Ak-quotient} is
\begin{equation*}
  -8k-2+3(2h+1+2k)-(2h-2k)=4h+1=n.
\end{equation*}
Thus $A_j(\alpha)/A_k(\alpha)=-\alpha^n=-1$, proving
\eqref{eq:Ak-pair}.
\qed

\noindent{\it Proof of Theorem~\ref{thm:guo72}.}
Put $n=4h+1$ and define
\begin{equation*}
  S_n(q)=\sum_{k=0}^{2h} [4k+1]
   \frac{(q;q^2)_k^3}{(q^2;q^2)_k^3}
   q^{k(n^2-2nk-n-2)/4}.
\end{equation*}
The exponent in each summand is an integer.  Indeed, since $n\equiv1\pmod4$,
\begin{equation*}
  k(n^2-2nk-n-2)\equiv -2k(k+1)\equiv0\pmod4.
\end{equation*}
The only possible denominator factors, apart from powers of $q$, are
$1-q^{2s}$ with $1\le s\le2h$, and these are coprime to $\Phi_n(q)$.

After substituting $q=x^2$, the exponent identity
\begin{align*}
 2\cdot\frac{k(n^2-2nk-n-2)}{4}
 &=\frac{k\bigl((4h+1)^2-2(4h+1)k-(4h+1)-2\bigr)}{2}\notag\\
 &=k\bigl(8h^2+2h-1-(4h+1)k\bigr)\notag\\
 &=-k+(4h+1)k(2h-k)
\end{align*}
shows that
\begin{equation*}
  S_n(x^2)=\sum_{k=0}^{2h} A_k(x)x^{nC_k},
  \qquad C_k=k(2h-k),
\end{equation*}
where $A_k(x)$ is defined by \eqref{eq:Ak-def}.

Note that
\begin{align*}
 S_n(x^2)&=\frac{1}{2}\left(\sum_{k=0}^{2h} A_k(x)x^{nC_k}+\sum_{k=0}^{2h} A_{2h-k}(x)x^{nC_k}\right)\notag\\
&=\frac{1}{2}\sum_{k=0}^{2h}\left(A_k(x)+A_{2h-k}(x)\right)x^{nC_k}.
\end{align*}
By Lemma~\ref{lem:thm1-pairing}, we have $A_k(x)+A_{2h-k}(x)\equiv 0\pmod{\Phi_n(x)}$ for $0\le k\le 2h$. It follows that
\begin{align}
 S_n(x^2)&\equiv\frac{1}{2}\sum_{k=0}^{2h}\left(A_k(x)+A_{2h-k}(x)\right)\notag\\
 &=\sum_{k=0}^{2h}A_k(x)\notag\\
 &\equiv 0 \pmod{\Phi_n(x)^2},\label{new-1}
\end{align}
where we have used Lemma \ref{prop:guo-aux} and the fact that $x^{nC_k}\equiv 1\pmod{\Phi_n(x)}$.

Finally, combining Lemma \ref{lem:squaring} and \eqref{new-1} completes the proof of Theorem~\ref{thm:guo72}.
\qed

\section{Proof of Theorem~\ref{thm:guo41}}\label{sec:proof-thm2}

Throughout this section let $d=2m+1>1$ be odd and let $\zeta$ be a primitive
$d$-th root of unity.  For $N\ge0$, set
\begin{equation*}
  H_N(q)=\frac{(q;q^2)_N}{(q^2;q^2)_N},\qquad
  K_N(q)=\frac{(q^2;q^4)_N}{(q^4;q^4)_N},
\end{equation*}
and
\begin{equation*}
  C_u=\prod_{j=0}^{u-1}\frac{2j+1}{2j+2}=\frac{(1/2)_u}{u!}\qquad(u\ge0).
\end{equation*}

\begin{lem}\label{lem:period}
Let $N=ud+t$, where $u\ge0$ and $0\le t<d$.  Then, as limits when
$q\to\zeta$,
\begin{equation}\label{eq:period-H}
H_N(\zeta)=
\begin{cases}
C_uH_t(\zeta),&0\le t\le m,\\
0,&m<t<d,
\end{cases}
\end{equation}
and
\begin{equation}\label{eq:period-K}
K_N(\zeta)=
\begin{cases}
C_uK_t(\zeta),&0\le t\le m,\\
0,&m<t<d.
\end{cases}
\end{equation}
In particular, these rational functions have finite limits at $\zeta$.
\end{lem}

\pf
We first treat $H_N$.  In a complete block $i=jd,jd+1,\ldots,jd+d-1$ of the
product defining $H_N$, the unique numerator zero occurs at $i=jd+m$ because
\begin{equation*}
  1+2(jd+m)=(2j+1)d,
\end{equation*}
whereas the unique denominator zero occurs at $i=jd+d-1$ because
\begin{equation*}
  2+2(jd+d-1)=(2j+2)d.
\end{equation*}
The ratio of the two zero factors is
\begin{equation*}
\lim_{q\to\zeta}\frac{1-q^{(2j+1)d}}{1-q^{(2j+2)d}}
=\frac{2j+1}{2j+2}.
\end{equation*}
After these two zero factors are removed, the remaining numerator and denominator
exponents run through the same nonzero residue classes modulo $d$; hence the
remaining products cancel at $q=\zeta$.  Thus the $j$-th complete block
contributes $(2j+1)/(2j+2)$.

The final incomplete block has length $t$.  If $t\le m$, it contains no zero
factor and contributes $H_t(\zeta)$.  If $t>m$, it contains the numerator zero
but no denominator zero, and the value is $0$.  This proves \eqref{eq:period-H}.

For $K_N$, the same argument applies.  In the $j$-th complete block the numerator
zero occurs at $i=jd+m$ and the denominator zero occurs at $i=jd+d-1$, since
\begin{equation*}
  2+4(jd+m)=(4j+2)d,
  \qquad
  4+4(jd+d-1)=(4j+4)d.
\end{equation*}
Their ratio is again $(2j+1)/(2j+2)$.  The remaining exponents run through the
same nonzero residue classes modulo $d$, because multiplication by $2$ permutes
$\Z/d\Z$.  The final incomplete block is handled exactly as above, which proves
\eqref{eq:period-K}.
\qed

\begin{lem}\label{lem:reflection}
For $0\le x\le m$,
\begin{equation}\label{eq:H-reflection}
  H_{m-x}(\zeta)=(-1)^m\zeta^{m^2+x}H_x(\zeta),
\end{equation}
and
\begin{equation}\label{eq:K-reflection}
  K_{m-x}(\zeta)=(-1)^m\zeta^{2m^2+2x}K_x(\zeta).
\end{equation}
Moreover, for $0\le r\le m$,
\begin{equation*}
  [4(m-r)+1]_{\zeta}=-\zeta^{-(4r+1)}[4r+1]_{\zeta}.
\end{equation*}
\end{lem}

\pf
No denominator factor in $H_x(\zeta)$, $H_{m-x}(\zeta)$, $K_x(\zeta)$, or
$K_{m-x}(\zeta)$ is zero.  We first record the product reversal used below.  Put
\[
  R_x(a,q)=\frac{(aq;q^2)_x}{(q^2/a;q^2)_x}.
\]
We claim that
\begin{equation}\label{eq:R-reflection}
  R_{m-x}(a,q)\equiv
  (-a)^{m-2x}q^{m^2+x}R_x(a,q)
  \pmod{\Phi_d(q)}.
\end{equation}
It is enough to prove this first for \(0\le x\le m/2\).  In that case, modulo
\(\Phi_d(q)\), the denominator factors may be reversed by replacing
\(i\) with \(m-1-i\):
\begin{align*}
\frac{R_{m-x}(a,q)}{R_x(a,q)}
&=\prod_{i=x}^{m-x-1}\frac{1-aq^{2i+1}}{1-q^{2i+2}/a}  \\
&\equiv
  \prod_{i=x}^{m-x-1}\frac{1-aq^{2i+1}}{1-q^{-(2i+1)}/a} \\
&=\prod_{i=x}^{m-x-1}(-a q^{2i+1})                         \\
&=(-a)^{m-2x}q^{m(m-2x)}
 \equiv (-a)^{m-2x}q^{m^2+x},
\end{align*}
because \(m(m-2x)-(m^2+x)=-xd\).  If \(x>m/2\), then applying the just-proved
case to \(m-x\) and taking reciprocals gives the same formula, since
\((m^2+x)+(m^2+m-x)=md\).  This proves \eqref{eq:R-reflection}.  Equivalently,
\begin{equation*}
\frac{(aq;q^2)_{m-x}}{(q^2/a;q^2)_{m-x}}
\equiv
(-a)^{m-2x}q^{m^2+x}
\frac{(aq;q^2)_x}{(q^2/a;q^2)_x}
\pmod{\Phi_d(q)}.
\end{equation*}
Taking \(a=1\) and \(q=\zeta\) gives \eqref{eq:H-reflection}.  Replacing
\(q\) by \(q^2\) in \eqref{eq:H-reflection} gives \eqref{eq:K-reflection},
because \(\zeta^2\) is again a primitive \(d\)-th root of unity.  Finally,
$4(m-r)+1=2d-(4r+1)$, and therefore
\begin{equation*}
[4(m-r)+1]_{\zeta}
=\frac{1-\zeta^{-(4r+1)}}{1-\zeta}
=-\zeta^{-(4r+1)}[4r+1]_{\zeta}.
\end{equation*}
This proves the lemma.
\qed

\begin{lem}\label{lem:finite}
Let $d>1$, and let $F(q)=P(q)/Q(q)\in\Q(q)$, where $P,Q\in\Q[q]$ are coprime.
If $F(q)$ has a finite limit at every primitive $d$-th root of unity, then
$\Phi_d(q)\nmid Q(q)$.  If, in addition, $F(\zeta)=0$ for every primitive
$d$-th root of unity $\zeta$, then $\Phi_d(q)\mid P(q)$.
\end{lem}

\pf
If $\Phi_d(q)$ divided $Q(q)$, then the finite-limit hypothesis would force
$P(q)$ to vanish at every primitive $d$-th root of unity.  Hence
$\Phi_d(q)$ would divide both $P(q)$ and $Q(q)$, contradicting their
coprimality.  Thus $\Phi_d(q)\nmid Q(q)$.  The second assertion follows at once:
then $Q(\zeta)\ne0$ for every primitive $d$-th root, and the vanishing of
$F(\zeta)$ for all such $\zeta$ implies that $P(q)$ is divisible by
$\Phi_d(q)$.
\qed

We now fix a divisor $d=2m+1$ of $n$, with $d>1$, and write
\begin{equation}\label{eq:s-adb}
  s=ad+b,
  \qquad 0\le b<d.
\end{equation}
For an integer $x$, let $\langle x\rangle_d$ denote its least nonnegative residue
modulo $d$.  By Lemma~\ref{lem:period}, the factor $H_N(\zeta)$ or $K_N(\zeta)$
vanishes whenever $\langle N\rangle_d>m$.  Hence
\begin{equation}\label{eq:T-nonzero-condition}
  \mathcal T_{s,k}(\zeta)=0
\end{equation}
unless
\begin{equation}\label{eq:three-residue-conditions}
\langle k-s\rangle_d\le m,
\qquad
\langle k\rangle_d\le m,
\qquad
\langle k+s\rangle_d\le m.
\end{equation}
If $r=\langle k\rangle_d$, then, since $s\equiv b\pmod d$, conditions
\eqref{eq:three-residue-conditions} become
\begin{equation}\label{eq:r-residue-conditions}
  \langle r-b\rangle_d\le m,\qquad
  r\le m,\qquad
  \langle r+b\rangle_d\le m .
\end{equation}
We claim that these conditions are equivalent to $r\in\mathcal R_b$, where
\begin{equation}\label{eq:Rb-def}
\mathcal R_b=
\begin{cases}
\{b,b+1,\ldots,m-b\},&
0\le b\le \lfloor m/2\rfloor,\\[2mm]
\{c,c+1,\ldots,m-c\},&
c=d-b,\ 1\le c\le \lfloor m/2\rfloor,\\[2mm]
\varnothing,&\text{otherwise}.
\end{cases}
\end{equation}
To prove the claim, first suppose that $0\le b\le m$.  Because the middle
condition in \eqref{eq:r-residue-conditions} gives $0\le r\le m$, the condition
\(\langle r-b\rangle_d\le m\) holds exactly when \(r\ge b\): if \(r<b\), then
\[
  \langle r-b\rangle_d=d+r-b\ge d-b=m+1+(m-b)>m,
\]
whereas if \(r\ge b\), then \(\langle r-b\rangle_d=r-b\le m\).  Also,
\(0\le r+b\le 2m=d-1\), so no reduction modulo \(d\) occurs in
\(\langle r+b\rangle_d\); hence
\[
  \langle r+b\rangle_d\le m
  \quad\Longleftrightarrow\quad
  r+b\le m
  \quad\Longleftrightarrow\quad
  r\le m-b .
\]
Thus in this case the three conditions are equivalent to
\[
  b\le r\le m-b.
\]
This gives the first line of \eqref{eq:Rb-def}; if \(b>\lfloor m/2\rfloor\),
the interval is empty.

It remains to consider the case \(b>m\).  Write \(b=d-c\), so
\(1\le c\le m\) and \(b\equiv -c\pmod d\).  Then
\eqref{eq:r-residue-conditions} is equivalent to
\[
  \langle r+c\rangle_d\le m,\qquad
  r\le m,\qquad
  \langle r-c\rangle_d\le m .
\]
As above, \(0\le r+c\le 2m=d-1\), so the first of these three conditions is
equivalent to \(r+c\le m\), that is, \(r\le m-c\).  The last condition is
equivalent to \(r\ge c\), since \(r<c\) would give
\(\langle r-c\rangle_d=d+r-c>m\), while \(r\ge c\) gives
\(\langle r-c\rangle_d=r-c\le m\).  Hence in this case the three conditions are
equivalent to
\[
  c\le r\le m-c.
\]
This gives the second line of \eqref{eq:Rb-def}; if
\(c>\lfloor m/2\rfloor\), the interval is empty.  These two cases exhaust
\(0\le b<d\), and the claim follows.

\begin{lem}\label{lem:block-cancel}
For every $\ell\ge0$,
\begin{equation}\label{eq:block-cancel}
  \sum_{j=0}^{d-1}\mathcal T_{s,s+\ell d+j}(\zeta)=0.
\end{equation}
\end{lem}

\pf
If $\mathcal R_b=\varnothing$, every term in the block vanishes by
\eqref{eq:T-nonzero-condition}, and there is nothing to prove.

Suppose first that $0\le b\le\lfloor m/2\rfloor$.  The possibly nonzero terms in
the block have
\begin{equation*}
  k=(a+\ell)d+r,
  \qquad b\le r\le m-b.
\end{equation*}
The reflection $r\mapsto r^*=m-r$ preserves this interval.  If $r=r^*$, then
$4r+1=d$, so $[4r+1]_{\zeta}=0$ and the corresponding summand is zero.  For
$r\ne r^*$ we compare the reflected term with the original one.  For
$k=(a+\ell)d+r$, the relevant indices are
\begin{equation*}
  k-s=\ell d+(r-b),
  \qquad
  k+s=(2a+\ell)d+(r+b),
  \qquad
  k=(a+\ell)d+r.
\end{equation*}
After reflection,
\begin{equation*}
  r^*-b=m-(r+b),
  \qquad
  r^*+b=m-(r-b).
\end{equation*}
Thus the complete-period constants from Lemma~\ref{lem:period} cancel in the
quotient of the reflected term by the original term.  The residual factors are
computed from Lemma~\ref{lem:reflection}:
\begin{equation*}
\frac{H_{r^*-b}(\zeta)H_{r^*+b}(\zeta)}
     {H_{r-b}(\zeta)H_{r+b}(\zeta)}
=\zeta^{2m^2+2r},
\end{equation*}
\begin{equation*}
\frac{H_{r^*}(\zeta)^2}{H_r(\zeta)^2}=\zeta^{2m^2+2r},
\qquad
\frac{K_{r^*}(\zeta)}{K_r(\zeta)}=(-1)^m\zeta^{2m^2+2r}.
\end{equation*}
Together with
\begin{equation*}
(-1)^{r^*-r}=(-1)^m,
\qquad
\frac{\zeta^{r^*}}{\zeta^r}=\zeta^{m-2r},
\qquad
\frac{[4r^*+1]_{\zeta}}{[4r+1]_{\zeta}}
=-\zeta^{-(4r+1)},
\end{equation*}
these identities give
\begin{align}
\frac{\mathcal T_{s,(a+\ell)d+r^*}(\zeta)}
     {\mathcal T_{s,(a+\ell)d+r}(\zeta)}
&=(-1)^m\left(-\zeta^{-(4r+1)}\right)
\zeta^{m-2r}\zeta^{2m^2+2r}
\zeta^{2m^2+2r}\left((-1)^m\zeta^{2m^2+2r}\right)\notag\\
&=-\zeta^{6m^2+m-1}=-1,
\label{eq:block-case1-ratio}
\end{align}
because
\begin{equation*}
  6m^2+m-1=(3m-1)(2m+1)=(3m-1)d.
\end{equation*}
Thus all nonzero terms cancel in pairs, with any fixed point equal to zero.

Now suppose $b=d-c$, where $1\le c\le\lfloor m/2\rfloor$.  The possibly nonzero
terms in the block have
\begin{equation*}
  k=(a+\ell+1)d+r,
  \qquad c\le r\le m-c.
\end{equation*}
Again $r^*=m-r$ preserves the interval, and a fixed point contributes zero.  For
$r\ne r^*$, the indices are
\begin{equation*}
  k-s=\ell d+(r+c),
  \qquad
  k+s=(2a+\ell+2)d+(r-c),
  \qquad
  k=(a+\ell+1)d+r.
\end{equation*}
After reflection,
\begin{equation*}
  r^*+c=m-(r-c),
  \qquad
  r^*-c=m-(r+c).
\end{equation*}
The complete-period constants again cancel, and the residual calculation is the
same as in \eqref{eq:block-case1-ratio}.  Therefore
\begin{equation*}
\mathcal T_{s,(a+\ell+1)d+r^*}(\zeta)
=-\mathcal T_{s,(a+\ell+1)d+r}(\zeta).
\end{equation*}
This proves \eqref{eq:block-cancel}.
\qed

\noindent{\it Proof of Theorem~\ref{thm:guo41}.}
Let
\begin{equation*}
  S_M(q)=\sum_{k=s}^{M}\mathcal T_{s,k}(q).
\end{equation*}
It is enough to prove $S_M(\zeta)=0$ for every primitive $d$-th root $\zeta$ with
$d\mid n$ and $d>1$.  Indeed, Lemma~\ref{lem:period} shows that every summand has
a finite limit at such a root of unity.  Lemma~\ref{lem:finite} then implies
that the reduced denominator of $S_M(q)$ is not divisible by $\Phi_d(q)$, and
that its numerator is divisible by $\Phi_d(q)$ once $S_M(\zeta)=0$ for all
primitive $d$-th roots $\zeta$.

Write
\begin{equation*}
  n=Nd,
\end{equation*}
where $N$ is odd, and keep the notation $s=ad+b$, $0\le b<d$, from
\eqref{eq:s-adb}.  Put
\begin{equation*}
  L=M-s.
\end{equation*}
Then
\begin{equation*}
  S_M(\zeta)=\sum_{j=0}^{L}\mathcal T_{s,s+j}(\zeta).
\end{equation*}
By Lemma~\ref{lem:block-cancel}, every complete block of $d$ consecutive values
of $j$ contributes zero.  Write
\[
  L=Qd+R,\qquad 0\le R<d.
\]
It remains only to examine the final block \(0\le j-Qd\le R\).

We next rewrite the nonzero condition in terms of the residue of the summation
shift \(j\).  In a fixed block, write
\[
  j=\ell d+t,\qquad 0\le t<d .
\]
Since \(k=s+j\) and \(s\equiv b\pmod d\), if \(r=\langle k\rangle_d\), then
\[
  r\equiv b+t\pmod d,
  \qquad\text{or equivalently}\qquad
  t=\langle r-b\rangle_d .
\]
Thus a residue \(t\) of \(j\) can give a nonzero summand only when
\(r\in\mathcal R_b\) in \eqref{eq:Rb-def}, and in that case
\(t=\langle r-b\rangle_d\).

Let \(\mathcal J_b\) denote this set of possible residues \(t=\langle j\rangle_d\).
It is obtained from \(\mathcal R_b\) by subtracting \(b\) modulo \(d\).  More
explicitly, if \(0\le b\le\lfloor m/2\rfloor\), then
\[
  r=b+u,\qquad 0\le u\le m-2b,
\]
and hence \(t=\langle r-b\rangle_d=u\).  If \(b=d-c\) with
\(1\le c\le\lfloor m/2\rfloor\), then
\[
  r\in\{c,c+1,\ldots,m-c\}
\]
and
\[
  t=\langle r-b\rangle_d=\langle r+c\rangle_d=r+c,
\]
because \(2c\le r+c\le m<d\).  Therefore the possible nonzero residues of \(j\)
are
\begin{equation}\label{eq:Jb-def}
\mathcal J_b=
\begin{cases}
\{0,1,\ldots,m-2b\},&
0\le b\le \lfloor m/2\rfloor,\\[2mm]
\{2c,2c+1,\ldots,m\},&
c=d-b,\ 1\le c\le \lfloor m/2\rfloor,\\[2mm]
\varnothing,&\text{otherwise}.
\end{cases}
\end{equation}
Thus only residues in \(\mathcal J_b\) can contribute to the last block.  If
\(\{0,1,\ldots,R\}\) contains all of \(\mathcal J_b\), then the last-block
contribution is the same as the full nonzero support of a complete block, and is
therefore zero by Lemma~\ref{lem:block-cancel}.  If \(\{0,1,\ldots,R\}\) is
disjoint from \(\mathcal J_b\), then every term in the last block is zero.

We now check the three choices of $M$ in \eqref{eq:M-values}.

First, let $M=(n-1)/2$.  Then
\begin{equation*}
L=\frac{n-1}{2}-s=\left(\frac{N-1}{2}-a\right)d+m-b.
\end{equation*}
The final residue bound is $R=m-b$ if $b\le m$, and $R=m+c$ if $b=d-c>m$.  In
the first case $R=m-b\ge m-2b$, so the prefix $\{0,1,\ldots,R\}$ contains the
whole first set in \eqref{eq:Jb-def}.  In the second case $R=m+c\ge m$, so it
contains the whole second set in \eqref{eq:Jb-def}.  Thus the final incomplete
block contributes zero.

Second, let $M=n-s-1$.  Then
\begin{equation*}
  L=n-2s-1=Nd-2ad-2b-1.
\end{equation*}
If $b=0$, the final residue bound is $R=d-1$, so the whole block is present.  If
$1\le b\le m$, then
\begin{equation*}
  R=d-2b-1=2m-2b\ge m-2b,
\end{equation*}
so the prefix contains the whole first set in \eqref{eq:Jb-def}.  If $b=d-c>m$,
then
\begin{equation*}
  R=2c-1,
\end{equation*}
whereas the second set in \eqref{eq:Jb-def}, when nonempty, begins at $2c$.
Thus the prefix contains none of the possible nonzero residues.  Again the final
contribution is zero.

Third, let $M=n-1$.  Then
\begin{equation*}
  L=n-s-1=Nd-ad-b-1.
\end{equation*}
If $b=0$, the final residue bound is $R=d-1$.  If $1\le b\le m$, then
\begin{equation*}
  R=d-b-1=2m-b\ge m-2b,
\end{equation*}
so the prefix contains the whole first set in \eqref{eq:Jb-def}.  If $b=d-c>m$,
then
\begin{equation*}
  R=c-1,
\end{equation*}
whereas the second set in \eqref{eq:Jb-def}, when nonempty, begins at $2c$.
Hence the prefix contains none of the possible nonzero residues.  The final
contribution is zero.

We have shown that $S_M(\zeta)=0$ for every primitive $d$-th root of unity with
$d\mid n$ and $d>1$.  Consequently $\Phi_d(q)$ divides the numerator of the
reduced form of $S_M(q)$ for every such $d$, while its denominator is relatively
prime to every such $\Phi_d(q)$.  Since the cyclotomic polynomials are pairwise
coprime and
\begin{equation*}
  [n]=\prod_{\substack{d\mid n\\ d>1}}\Phi_d(q),
\end{equation*}
the reduced numerator of $S_M(q)$ is divisible by $[n]$ and the reduced
denominator is relatively prime to $[n]$.  This proves \eqref{eq:thm2-main}.
\qed

\end{document}